\documentclass[oneside,10pt]
{amsart}          
\usepackage{graphicx}
\usepackage{amsfonts,
amsmath,latexsym,amssymb} 
\usepackage{amsthm}                
\usepackage{calrsfs}
\usepackage{enumerate}
\usepackage{enumitem}
\usepackage{url}
\usepackage[final]
{showkeys}

\usepackage{hyperref}

\newtheorem{theorem}{Theorem}

\theoremstyle{definition}

\newtheorem{definition}{Definition}

\newlist{thmRoman}{enumerate}{1}
\setlist[thmRoman]{%
  label=\normalfont
  \small(\Roman*), 
  ref=(\Roman*),                                    
  font=\normalfont
  \small,           
  align=left, leftmargin=*,                         
}

\newcommand{\al}{\alpha}

\renewcommand{\Psi}{\overline{\Phi}}

\newcommand{\R}{\mathbb{R}}
\newcommand{\C}{\mathbb{C}}

\newcommand{\F}
{\Sigma}

\renewcommand{\le}{\leqslant}
\renewcommand{\ge}{\geqslant}
\newcommand\oa{\mathsf{oa}}
\newcommand\ta{\mathsf{ta}}
\newcommand\ua{\mathsf{ua}}

\begin{document}

\title
{On the sum of the angles between three vectors
}


\author{Iosif Pinelis}

\address{Department of Mathematical Sciences\\
Michigan Technological University\\
Houghton, Michigan 49931, USA\\
}

%



\date{5 August 2015}                               


        



\begin{abstract}
For any three nonzero vectors $a,b,c$ in $\R^2$, we obtain a necessary and sufficient condition for the sum 
of the three pairwise angles between these vectors to equal $2\pi$. 
As an easy consequence of this, a proof of Euclid's theorem that the sum of the interior angles of any triangle is $\pi$ is provided. So, the main result of this note can be considered a generalization of Euclid's theorem. 
To a large extent, the consideration is reduced almost immediately to a choice for the sum of three related angles among the three integer multiples $0,2\pi,4\pi$ of $\pi$. The rest of the consideration concerns only various betweenness relations. 
\end{abstract}

\maketitle







\section{Definitions of relevant notions, their properties, and main results}\label{result}

Let us begin with the definitions of several kinds of angles or, more exactly, angle measures. 

Let $a,b,c$ be nonzero vectors in $\R^2$. We identify $\R^2$ with the set $\C$ of all complex numbers. 

\begin{definition}\label{def:oa}
The \emph{oriented angle} $\oa_{a,b}$ from $a$ to $b$ is defined as the value in $[0,2\pi)$ of the argument of the ratio $b/a$. That is, $\oa_{a,b}=t$ if $b/a=re^{it}$ for some real $r>0$ and some real $t\in[0,2\pi)$. 
\end{definition}

\begin{definition}\label{def:ta}
The \emph{turning angle} $\ta_{a,b}$ from $a$ to $b$ is defined as \break  $\min(\oa_{a,b},2\pi-\oa_{a,b})$.  
\end{definition}

\begin{definition}\label{def:ua}
The \emph{usual angle} $\ua_{a,b}$ from $a$ to $b$ is defined as 
$\arccos\frac{a\cdot b}{|a|\,|b|}$, where $a\cdot b$ is the dot product of $a$ and $b$.  
\end{definition}

The following is the main result of this note. 

\begin{theorem}\label{th:ta}
$\ta_{a,b}+\ta_{b,c}+\ta_{c,a}=2\pi$ if and only if one of the following two  alternatives holds: 
\begin{thmRoman}
	\item\label{alt1} $0<\max(\oa_{a,b},\oa_{b,c},\oa_{c,a})\le\pi$;
	\item\label{alt2} $\min(\oa_{a,b},\oa_{b,c},\oa_{c,a})\ge\pi$. 
\end{thmRoman}	 
\end{theorem}

The two alternatives in Theorem~\ref{th:ta} are mutually exclusive. Indeed, if both \break 
alternatives \ref{alt1} and \ref{alt2} hold, then $\oa_{a,b}=\oa_{b,c}=\oa_{c,a}=\pi$, which contradicts Property~\ref{i} (stated on p.\ \pageref{i}).

That the two alternatives in Theorem~\ref{th:ta} are not quite the mirror images of each other reflects the asymmetry of the right-open interval $[0,2\pi)$, which is the range of values of oriented angles.


\begin{figure}[h]
	\centering 		\includegraphics[width=0.80\textwidth]{
	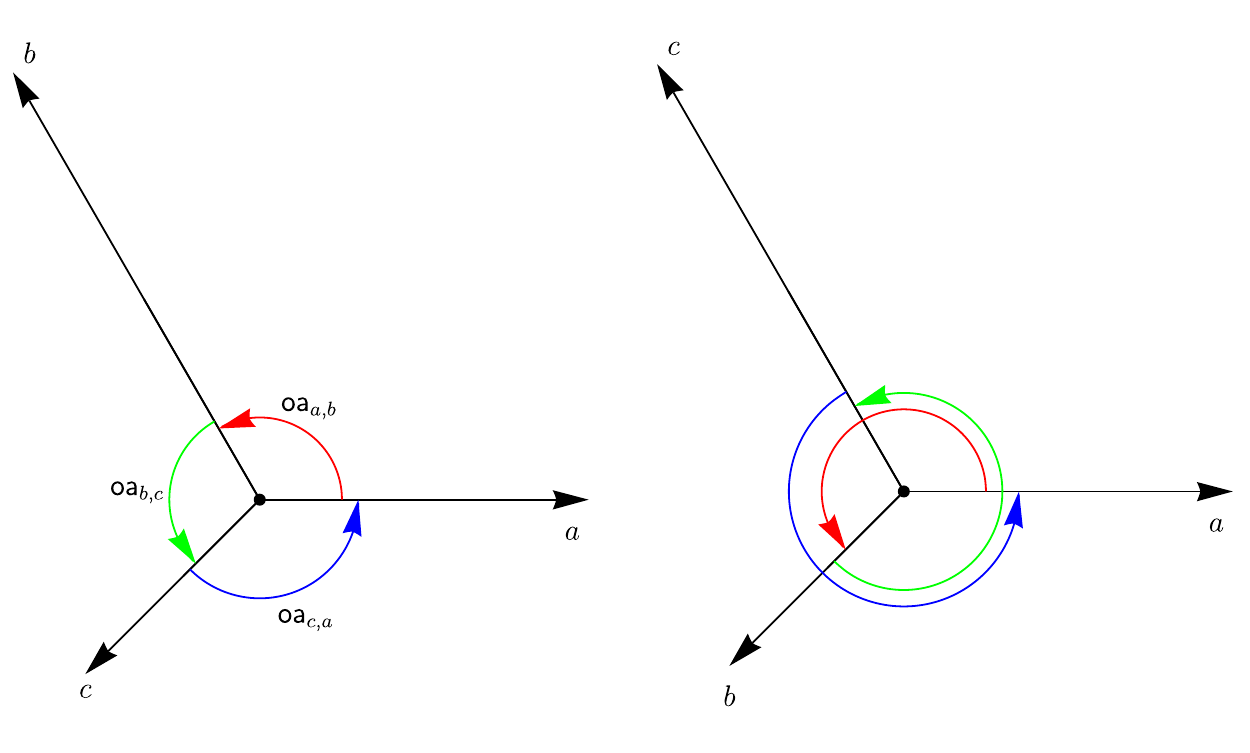}
	\caption{Two alternatives in the necessary and sufficient 
	condition in Theorem~\ref{th:ta}: \ref{alt1} (left) and \ref{alt2} (right).}
	\label{fig:pic}
\end{figure}

These two alternatives, \ref{alt1} and \ref{alt2}, are illustrated in Fig.~\ref{fig:pic}, where the oriented angles $\oa_{a,b}$, $\oa_{b,c}$, and $\oa_{c,a}$ are shown in red, green, and blue, respectively. 

The left picture in Fig.~\ref{fig:pic} illustrates alternative \ref{alt1} in Theorem~\ref{th:ta}. There all the oriented angles $\oa_{a,b},\oa_{b,c},\oa_{c,a}$ are in the interval $(0,\pi)$ and therefore coincide with the corresponding turning angles $\ta_{a,b},\ta_{b,c},\ta_{c,a}$, which latter thus ``obviously'' sum to $2\pi$.  

The right picture in Fig.~\ref{fig:pic} illustrates alternative \ref{alt2} in Theorem~\ref{th:ta}. There all the oriented angles $\oa_{a,b},\oa_{b,c},\oa_{c,a}$ are in the interval $(\pi,2\pi)$ and therefore the corresponding turning angles $\ta_{a,b},\ta_{b,c},\ta_{c,a}$ are $2\pi-\oa_{a,b},2\pi-\oa_{b,c},2\pi-\oa_{c,a}$. In this case, $\oa_{a,b}+\oa_{b,c}+\oa_{c,a}=4\pi$ and hence $\ta_{a,b}+\ta_{b,c}+\ta_{c,a}=3\times2\pi-4\pi=2\pi$.  In the right picture, the oriented angles $\oa_{a,b},\oa_{b,c},\oa_{c,a}$ ``overlap'' and therefore their labels are not shown, in contrast with the left picture.



\medskip
\hrule
\medskip

Let us precede the proof of Theorem~\ref{th:ta} by the following: 

\medskip

\noindent{\bf Properties} of oriented, turning, and usual angles:
\begin{enumerate}
[label=(\roman*)]
	\item\label{i} 
	$\oa_{a,b}+\oa_{b,c}+\oa_{c,a}\in\{0,2\pi,4\pi\}$. 
	\item\label{i'} 
	$\oa_{b,a}=2\pi-\oa_{a,b}$ -- unless $\oa_{a,b}=0$, in which case $\oa_{b,a}=\oa_{a,b}=0$. 
	\item \label{ii} $\ta_{a,b}=\ua_{a,b}\in[0,\pi]$. 
	\item \label{iii}  $\ta_{b,a}=\ta_{a,b}$. 
\end{enumerate} 

\begin{proof}[Proof of Property~\emph{\ref{i}}] Note that 
\begin{equation*}
	1=\frac ba\,\frac cb\,\frac ac
	=\frac{|b|}{|a|}e^{i\,\oa_{a,b}}\,\frac{|c|}{|b|}e^{i\,\oa_{c,b}}\,
	\frac{|a|}{|c|}e^{i\,\oa_{c,a}}=e^{i(\oa_{a,b}+\oa_{b,c}+\oa_{c,a})}. 
\end{equation*}
So, $\oa_{a,b}+\oa_{b,c}+\oa_{c,a}$ is an integer multiple of $2\pi$. It remains to recall that each of the values $\oa_{a,b},\oa_{b,c},\oa_{c,a}$ is in the interval $[0,2\pi)$. 
\end{proof}

\begin{proof}[Proof of Property~\emph{\ref{i'}}] This follows immediately from Definition~\ref{def:oa}. 
\end{proof}

\begin{proof}[Proof of Property~\emph{\ref{ii}}] 
Letting $\bar z$ denote the complex conjugate of a complex number $z$, we have  
\begin{equation*}
\cos\ua_{a,b}=	\frac{a\cdot b}{|a|\,|b|}=\frac{\Re(\bar a b)}{|a|\,|b|}
	=\frac{|a|^2\Re(b/a)}{|a|\,|b|}
	=\frac{|a|^2 \frac{|b|}{|a|} \cos\oa_{a,b} }{|a|\,|b|}
	=\cos\oa_{a,b}=\cos\ta_{a,b}.  
\end{equation*}
It remains to note that both $\ta_{b,a}$ and $\ua_{a,b}$ are in the interval $[0,\pi]$, which follows immediately from Definitions~\ref{def:ta} and \ref{def:ua}. 
\end{proof}

\begin{proof}[Proof of Property~\emph{\ref{iii}}] This follows immediately from Definition~\ref{def:ta} and Property~\ref{i'}.
\end{proof}

Now we can provide

\begin{proof}[Proof of Theorem~\ref{th:ta}] If alternative \ref{alt1} holds,  
then, by Definition~\ref{def:ta} and Property~\ref{i},  
\begin{equation*}
	\ta_{a,b}+\ta_{b,c}+\ta_{c,a}
	=\oa_{a,b}+\oa_{b,c}+\oa_{c,a}\in\{2\pi,4\pi\}.
\end{equation*}
By Property~\ref{ii}, $\ta_{a,b}+\ta_{b,c}+\ta_{c,a}\in[0,3\pi]$. So, $\ta_{a,b}+\ta_{b,c}+\ta_{c,a}=2\pi$. 

Similarly, if 
alternative \ref{alt2} holds, then   
\begin{equation*}
	\ta_{a,b}+\ta_{b,c}+\ta_{c,a}
	=3\times2\pi-(\oa_{a,b}+\oa_{b,c}+\oa_{c,a})\in\{6\pi,4\pi,2\pi\}.
\end{equation*}
Therefore and because 
$\ta_{a,b}+\ta_{b,c}+\ta_{c,a}\in[0,3\pi]$, we have  $\ta_{a,b}+\ta_{b,c}+\ta_{c,a}=2\pi$. 


This completes the proof of the ``if'' part of Theorem~\ref{th:ta}. 

To obtain a contradiction and thus prove the ``only if'' part of Theorem~\ref{th:ta}, suppose that 
$\ta_{a,b}+\ta_{b,c}+\ta_{c,a}=2\pi$ whereas 
neither alternative \ref{alt1} nor alternative \ref{alt2} holds. Then either  
\begin{itemize}
	\item $\max(\oa_{a,b},\oa_{b,c},\oa_{c,a})=0$ or 
	\item $\max(\oa_{a,b},\oa_{b,c},\oa_{c,a})>\pi$ and $\min(\oa_{a,b},\oa_{b,c},\oa_{c,a})<\pi$.
\end{itemize} 
By cyclic symmetry, without loss of generality (wlog) $\max(\oa_{a,b},\oa_{b,c},\oa_{c,a})=\oa_{c,a}$ and $\min(\oa_{a,b},\oa_{b,c},\oa_{c,a})=\oa_{a,b}$.
So, wlog one of the following three cases must occur: 
\begin{enumerate}
[label=\emph{Case}~\arabic*:]
	\item $\oa_{a,b}=\oa_{b,c}=\oa_{c,a}=0$;
	\item $\oa_{a,b}<\pi<\oa_{c,a}$ and $\oa_{b,c}\le\pi$;
	\item $\oa_{a,b}<\pi<\oa_{c,a}$ and $\oa_{b,c}>\pi$.
\end{enumerate}

In Case~1, we have $2\pi=\ta_{a,b}+\ta_{b,c}+\ta_{c,a}=\oa_{a,b}+\oa_{b,c}+\oa_{c,a}=0$, a contradiction. 

In Case~2, 
by Definition~\ref{def:ta}, $2\pi=\ta_{a,b}+\ta_{b,c}+\ta_{c,a}= 
\oa_{a,b}+\oa_{b,c}+2\pi-\oa_{c,a}$, so that $\oa_{a,b}+\oa_{b,c}=\oa_{c,a}$. So, by Property~\ref{i}, 
$2\oa_{c,a}=\oa_{a,b}+\oa_{b,c}+\oa_{c,a}\in\{0,2\pi,4\pi\}$, whence
$\oa_{c,a}\in\{0,\pi,2\pi\}$, which contradicts the condition $\pi<\oa_{c,a}$ of Case~1, since, by Definition~\ref{def:oa}, $\oa_{c,a}<2\pi$. 

Finally, in Case~3, 
again by Definition~\ref{def:ta}, $2\pi=\ta_{a,b}+\ta_{b,c}+\ta_{c,a}= 
\oa_{a,b}+2\pi-\oa_{b,c}+2\pi-\oa_{c,a}$, so that $\oa_{b,c}+\oa_{c,a}=\oa_{a,b}+2\pi$. Using now Property~\ref{i} again, we get $2\oa_{a,b}+2\pi=\oa_{a,b}+\oa_{b,c}+\oa_{c,a}\in\{0,2\pi,4\pi\}$. Therefore and because $\oa_{a,b}\in[0,2\pi)$, we get $\oa_{a,b}\in\{0,\pi\}$. So, by the Case~3 condition, $\oa_{a,b}=0$, whence $\oa_{b,c}+\oa_{c,a}=\oa_{a,b}+2\pi=2\pi$, which contradicts the Case~3 conditions on $\oa_{c,a}$ and $\oa_{b,c}$. 


This completes the proof of the ``only if'' part of Theorem~\ref{th:ta} as well. 

Thus, Theorem~\ref{th:ta} is proved. 
\end{proof}

Take now any pairwise distinct points $p_0,p_1,p_2$ in $\R^2$. For any integer $i$, let $p_i:=p_r$, where $r$ is the remainder of the division of $i$ by $3$, and consider the interior angle 
\begin{equation*}
	\al_i:=\ua_{p_{i-1}-p_i,\,p_{i+1}-p_i}
\end{equation*}
at vertex $p_i$ of the triangle $p_0 p_1 p_2$, with the usual angle $\ua_{a,b}$ as in Definition~\ref{def:ua}.

\begin{theorem}\label{th:}
$\al_0+\al_1+\al_2=\pi$. 
\end{theorem}

\begin{proof}
Consider the nonzero vectors $a:=p_1-p_0,\quad b:=p_2-p_1,\quad c:=p_0-p_2$. Then $a+b+c=0$, so that 
\begin{equation*}
	0=\Im\Big(1+\frac ba+\frac ca\Big)
	=\frac{|b|}{|a|}\,\sin\oa_{a,b}-\frac{|c|}{|a|}\,\sin\oa_{c,a}. 
\end{equation*}
So, $\sin\oa_{a,b}$ and $\sin\oa_{c,a}$ (and, similarly, $\sin\oa_{b,c}$) are of the same sign. So, either $\max(\oa_{a,b},\oa_{b,c},\oa_{c,a})\le\pi$ or $\min(\oa_{a,b},\oa_{b,c},\oa_{c,a})\ge\pi$. Moreover, \break  $\max(\oa_{a,b},\oa_{b,c},\oa_{c,a})>0$, because otherwise $\oa_{a,b}=\oa_{b,c}=\oa_{c,a}=0$ and hence $0=1+\frac ba+\frac ca
=1+\frac{|b|}{|a|}+\frac{|c|}{|a|}>1$, a contradiction. So, one of the two alternatives in Theorem~\ref{th:ta}, \ref{alt1} or \ref{alt2}, holds. 
Using also Property~\ref{ii}, we conclude that 
\begin{equation*}
\ua_{a,b}+\ua_{b,c}+\ua_{c,a}=\ta_{a,b}+\ta_{b,c}+\ta_{c,a}=2\pi. 
\end{equation*}
It remains to note that $\al_0=\ua_{-c,a}=\pi-\ua_{c,a}$, $\al_1=\ua_{-a,b}=\pi-\ua_{a,b}$, and $\al_2=\ua_{-b,c}=\pi-\ua_{b,c}$, so that 
$\al_0+\al_1+\al_2=3\pi-(\ua_{a,b}+\ua_{b,c}+\ua_{c,a})=\pi$. 
\end{proof}

\section{Discussion}\label{discussion}

According to Theorem~\ref{th:ta}, a necessary and sufficient condition for 
\begin{equation}\label{eq:=2pi}
	\ta_{a,b}+\ta_{b,c}+\ta_{c,a}=2\pi 
\end{equation} 
to hold is 
a disjunction of conjunctions of inequalities for the oriented angles \break  $\oa_{a,b},\oa_{b,c},\oa_{c,a}$. Let us emphasize that Theorem~\ref{th:ta} holds for all nonzero vectors $a,b,c$, which do not have to be the ``side-vectors'' of a triangle. In particular, this is clearly seen in Fig.~\ref{fig:pic}, where the vectors $a,b,c$ can be of any nonzero lengths (so that e.g.\ $|c|>|a|+|b|$). So, the restriction that vectors $a,b,c$ be the ``side-vectors'' of a triangle is largely irrelevant.  

A very simple but crucial observation that is the main ingredient in the proof of Theorem~\ref{th:ta} is Property~\ref{i} (stated on p.\ \pageref{i}), which in a sense reduces the entire consideration to a choice among 
just a few integer multiples of $\pi$.

On the other hand, if -- 
as in the proof of Theorem~\ref{th:} -- $a,b,c$ \emph{are} the nonzero ``side-vectors'' of a triangle, so that $a+b+c=0$, then, as was shown at the end of that proof, the statement of Theorem~\ref{th:} is equivalent to \eqref{eq:=2pi}. It is also seen from the proof of Theorem~\ref{th:} that the ``triangle'' condition $a+b+c=0$ was used there to a rather limited extent -- just to obtain the desired inequalities in 
one of the alternatives \ref{alt1} or \ref{alt2} in Theorem~\ref{th:ta}. 

So, it appears that 
Theorem~\ref{th:} is 
mainly about \emph{inequalities} (between angles) or, in other words, about \emph{betweenness}. 

The notion of betweenness is the subject of ``Group II: Axioms of Order'' in Hilbert's \emph{The Foundations of Geometry} \cite{hilbert02}. However, concerning ``T\textsc{heorem}~19'' and ``T\textsc{heorem}~20. The sum of the angles of a triangle is two right angles.'' in that book, Hilbert only says ``we can then
easily establish [these] propositions''. 

As noted e.g.\ by Greenberg~\cite[p.~104]{greenberg08} concerning \cite{euclid}, ``Euclid never mentioned this notion [of betweenness] explicitly
but tacitly assumed certain facts about it that seem obvious in 
diagrams.'' 


However, 
the 6-line proof of the sum-of-angles theorem in 
\cite
{greenberg08} -- presented there as Proposition~4.11 -- also refers to a diagram, without mentioning any betweenness. 
Surprisingly, it appears that there is no published complete synthetic proof of this ancient theorem, with the betweenness relations explicitly and adequately addressed -- cf.\ the discussion at \url{https://mathoverflow.net/q/498127/36721}.

\bibliographystyle{abbrv}


\bibliography{C:/Users/ipinelis/Documents/pCloudSync/mtu_pCloud_02-02-17/bib_files/citations04-02-21}


\end{document}